\documentclass{amsart}
\usepackage{amsmath,amssymb,bm}
\usepackage[dvips]{graphicx}
\usepackage{color}
\theoremstyle{definition}
\newtheorem{defn}{Definition}[section]

\newtheorem{rem}[defn]{Remark}
\theoremstyle{plain}
\newtheorem{thm}[defn]{Theorem}

\newtheorem{lem}[defn]{Lemma}

 \makeatletter
    
    \@addtoreset{equation}{section}
    \@namedef{subjclassname@2020}{%
  \textup{2020} Mathematics Subject Classification}
  \makeatother
\title[]{A note on stabilization heights of fiber surfaces and the Hopf invariants}
\author{Keiji Tagami}
\address{Department of Fisheries Distribution and Management, National Fisheries University, Shimonoseki, Yamaguchi 759-6595 JAPAN}
\email{tagami@fish-u.ac.jp}
\date{\today}
\begin{document}
\begin{abstract}
In this paper, we focus on the Hopf invariant and 
give an alternative proof for the unboundedness of stabilization heights of fiber surfaces, which was firstly proved by Baader and Misev. 
\end{abstract}
\maketitle
\section{Introduction}
A plumbing is one of well-known operations gluing two surfaces (\cite{Stallings}). 
If one of the surfaces is a Hopf band, we call the plumbing a {\it Hopf plumbing}. 
The inverse operation of a Hopf plumbing is called a {\it Hopf deplumbing} (see Figure~\ref{fig:plumbing}). 
It is known that a surface obtained by plumbing (more strongly, Murasugi sum) two surfaces is fibered if and only if the two surfaces are fibered (see \cite{gabai2, gabai1}). In particular, Hopf plumbings preserve the fiberedness of surfaces. 
\begin{figure}[h]
\centering
\includegraphics[scale=0.5]{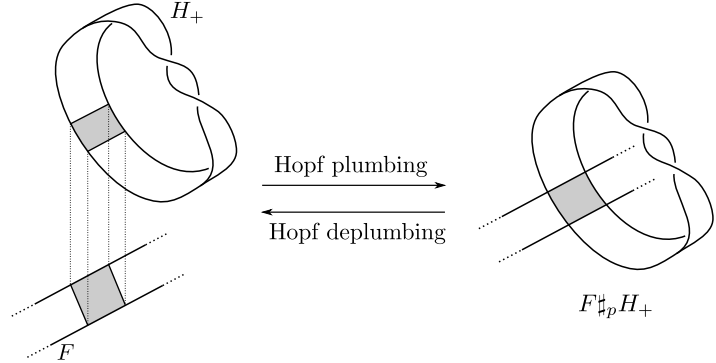}
\caption{Schematic picture of Hopf plumbing and Hopf deplumbing. Here, we only draw the case the Hopf band is the positive Hopf band $H_{+}$. In this picture, $F$ and $H_{+}$ are split. }
\label{fig:plumbing}
\end{figure}
\par
Harer \cite{harer2, Harer1} proved that any fiber surface in $\mathbf{S}^{3}$ is obtained from the standard disk by a sequence of Hopf plumbings, Hopf deplumbings and twistings 
(for definition of twisting, see Section~\ref{sec:stallings}). 
Harer asked whether we can omit twistings. 
Giroux \cite{Giroux1} and Goodman \cite{Goodman1}, independently, gave the affirmative answer, that is, every fiber surface is obtained from the standard disk by a finite sequence of Hopf plumbings, followed by a finite sequence of Hopf deplumbings (see also \cite{Giroux-Goodman}). 
Unfortunately, there is a fiber surface which cannot be constructed from the standard disk by only Hopf plumbings. 
For example, see Melvin and Morton's work \cite{Melvin-Morton}. 
\par 
Baader and Misev \cite{Baader-Misev} studied the minimal number of Hopf deplumbings which are needed to obtain a given fiber surface $\Sigma$ from the standard disk $D$. 
It is called the {\it stabilization height} of $\Sigma$. 
More precisely, the stabilization height of a fiber surface $\Sigma$ is defined as follows. 
Let $S$ be a fiber surface which is obtained from both $\Sigma$ and the standard disk $D$ by a finite sequence of Hopf plumbings (without Hopf deplumbings). 
Such a surface $S$ is called a {\it common stabilization} of $\Sigma$ and $D$. 
Then, the stabilization height $h(\Sigma)$ of a fiber surface $\Sigma$ is defined by 
\[
h(\Sigma)=\min\{b_{1}(S)-b_{1}(\Sigma)\mid S \text{ is a common stabilization of }\Sigma \text{ and }D\}, 
\]
where $b_1$ denotes the first Betti number. 
Baader and Misev proved that the stabilization height is unbounded. 
\begin{thm}[{\cite[Theorem~1.1]{Baader-Misev}}]\label{thm:baader-misev}
Let $\Sigma_{n}\subset \mathbf{S}^3$ be a family of fiber surfaces of the same topological type $\Sigma$ whose monodromies $\varphi_n\colon \Sigma \rightarrow \Sigma$ differ by a power 
of a Dehn twist $t_c$ along an essential simple closed curve $c\subset \Sigma$ $\colon \varphi_n=\varphi_0\circ t^{n}_c$. Then 
\[
\lim_{|n|\rightarrow \infty}h(\Sigma_n)=\infty. 
\]
\end{thm}
Roughly speaking, to prove Theorem~\ref{thm:baader-misev}, Baader and Misev showed the following (a) and (b).    
\begin{itemize}
\item[(a)] $\displaystyle{\lim_{|n|\rightarrow \infty}}\operatorname{scl}(\varphi_n)=\infty$, where $\operatorname{scl}(\varphi_n)$ is the stable commutator length of $\varphi_n$ in the mapping class group of a closed surface obtained from a common stabilization $S_n$ of $\Sigma_{n}$ and $D$ 
by capping off $S_n$ by a disk after plumbing at most six Hopf bands so that $g(S_{n})\geq 3$, 
\item[(b)] $\operatorname{scl}(\varphi_n)<C(b_{1}(S_n))$, where $C(b_{1}(S_n))$ is a constant depending only on $b_{1}(S_n)$. 
From (a), we obtain $C(b_{1}(S_n))\rightarrow \infty$ as $|n|\rightarrow \infty$ and this implies $\displaystyle{\lim_{|n|\rightarrow \infty}}b_{1}(S_n)=\infty$. 
\end{itemize}
\par 
In this paper, we give a different approach to the unboundedness of stabilization heights. 
In particular, we prove the following. 
\begin{thm}\label{thm:main}
There are a family $\{\Sigma_{n}\}_{n\in\mathbf{Z}}$ of fiber surfaces in $\mathbf{S}^3$ of the same topological type $\Sigma$ and a constant $C\in\mathbf{Z}_{>0}$ such that 
\begin{itemize}
\item $\displaystyle{\lim_{|n|\rightarrow \infty}}h(\Sigma_n)=\infty$, 
\item $\operatorname{scl}(\psi_n) <C $ for any $n\in \mathbf{Z}$, where $\psi_{n}$ is the monodromy of $\Sigma_{n}$ and we consider the stable commutator length $\operatorname{scl}(\psi_n)$ in the mapping class group of any surface containing $\Sigma$ as a subsurface. 
\end{itemize}
\end{thm}
We prove Theorem~\ref{thm:main} in Section~\ref{sec:proof}. 
To prove the unboundedness of stabilization heights, we use the plane fields on $\mathbf{S}^3$ obtained from $\Sigma_{n}$ by Thurston and Winkelnkemper's consruction \cite{Thurston-Winkelnkemper} and homotopy invariants for the plane fields. 
In particular, we use the Hopf invariant.  
\begin{rem}
The family $\{\Sigma_{n}\}_{n\in\mathbf{Z}}$ in Theorem~\ref{thm:main} does not satisfy the condition of Theorem~\ref{thm:baader-misev} because of the boundedness of the stable commutator length and $(a)$. In particular, the technique to prove Theorem~\ref{thm:baader-misev} cannot be used to prove Theorem~\ref{thm:main}. 
Conversely, there is some example satisfying the condition of Theorem~\ref{thm:baader-misev}  where we cannot use the technique to prove Theorem~\ref{thm:main} (for more detail, see Section~\ref{sec:discussion}). 
\end{rem}
\par
This paper is organized as follows. 
In Section~\ref{sec:stallings}, we recall the definition of Stallings twists. 
In Section~\ref{sec:Hopf}, we recall properties of the Hopf invariant and stabilization heights. 
In Section~\ref{sec:proof}, we prove Theorem~\ref{thm:main}. 
In Section~\ref{sec:discussion}, we compare Baader and Misev's work with ours. 
\section{Stallings twist}\label{sec:stallings}
\subsection{Open book decompositions}\
Let $\Sigma$ be an oriented surface with boundary and $f\colon \Sigma\rightarrow \Sigma$ be a diffeomorphism on $\Sigma$ fixing the boundary. 
Consider the pinched mapping torus 
\[\widehat{M}_{f}=\Sigma\times [0,1]/_{\sim},\]
where the equivalent relation $\sim$ is defined as follows:
\begin{itemize}
\item  $(x,1)\sim (f(x),0)$ for $x \in \Sigma$, and 
\item  $(x,t)\sim (x,t')$ for $x \in \partial \Sigma$ and $t$, $t'$ $\in [0,1]$. 
\end{itemize}
Here, we orient $[0,1]$ from $0$ to $1$ and we give an orientation of $\widehat{M}_{f}$ by the orientations of $\Sigma$ and $[0,1]$.
Let $M$ be a closed  oriented $3$-manifold.
If $M$ is orientation-preserving  diffeomorphic to $\widehat{M}_{f}$, 
the pair $(\Sigma, f)$ is called an \textit{open book decomposition} of $M$. 
The map $f$ is called the \textit{monodromy} of $(\Sigma, f)$. 
The boundary $L=\partial \Sigma$ of $\Sigma$ is called a \textit{fibered link} in $M$, and $\Sigma\subset M$ is called a \textit{fiber surface} of $L$. 
The diffeomorphism $f$ is also called the \textit{monodromy of $L$}. 
\subsection{Twistings and Stallings twists}\
Let $M$ be a closed oriented $3$-manifold, and $(\Sigma, f)$ be an open book decomposition of $M$.
Let $c$ be a simple closed curve on  a fiber surface  $\Sigma \subset M$.
Then,  {\it a twisting along $c$ of order $n$} is defined as performing $(1/n)$-surgery along $c$ with respect to the framing determined by $\Sigma$. 
Then we obtain the following.
\begin{lem}[Stallings \cite{Stallings}]\label{lem:Stallings}
Let $M$ be a closed oriented $3$-manifold and $(\Sigma, f)$ be an open book decomposition of $M$.
Then, the resulting manifold obtained from $M$ by a twisting along $c$ of order $n$
is orientation-preservingly diffeomorphic to $\widehat{M}_{t_{c}^{-n}\circ f}$.
\end{lem}
Lemma~\ref{lem:Stallings} implies that, by a  twisting along $c$ of order $n$,
the fibered link with monodromy $f$ 
is changed into the fibered link with monodromy $t_{c}^{-n}\circ f$. 
\par
Consider the case $M=\mathbf{S}^{3}$. 
Let $(\Sigma, f)$ be an open book decomposition of $\mathbf{S}^{3}$. 
Let $c\subset \Sigma$ be an essential simple closed curve on $\Sigma$. 
Suppose that $c$ is the unknot in $\mathbf{S}^{3}$ and the framing of $c$ determined by $\Sigma$ is equal to the framing of $c$ determined by a disk bounded by $c$. 
We call such a curve $c$ a {\it twisting loop}. 
Then, the twisting along a twisting loop $c$ of order $n$ is equivalent to performing $(1/n)$-surgery along the unknot, where $n\in\mathbf{Z}$. 
In this paper, we call this twisting the {\it Stallings twist} along $c$ of order $n$ 
(for a pictorial description of a Stallings twist, see \cite[Figure~1]{Yamamoto2}). 
Remark that this definition is slightly different from the original definition given by Stallings \cite{Stallings}. 
It is easy to see that a Stallings twist preserves the topology of $\mathbf{S}^3$. 
\par
Let $\Sigma\subset \mathbf{S}^{3}$ be a fiber surface.  
Let $c\subset \Sigma$ be a twisting loop and $D$ be a disk bounded by $c$. 
Suppose that the minimal number of components of the intersection between $\operatorname{Int}(D)$ and $\Sigma$ is equal to $m$. 
Then, the Stallings twist along $c$ is called {\it type $(0,m)$} 
\footnote{The first ``$0$" of $(0,m)$ represents that the surface framing of $c\subset \Sigma$ is equal to the ``$0$"-framing. }  (see \cite{Yamamoto1}). 
\par
By Giroux and Goodman's work, Stallings twists can be realized by some Hopf plumbings and deplumbings. 
For example, Yamamoto \cite{Yamamoto1} proved the following. 
\begin{thm}[{\cite[Theorem~1.1]{Yamamoto1}}]\label{thm:yamamoto}
A Stallings twist of type $(0,1)$ can be realized by one Hopf plumbing and one Hopf deplumbing. 
Moreover, the Hopf bands have the same sign, that is, both of them are positive Hopf bands or negative Hopf bands. 
\end{thm}
\begin{rem}\label{rem:Yamamoto}
The original statement of \cite[Theorem~1.1]{Yamamoto1} does not contain the last part of Theorem~\ref{thm:yamamoto}. 
However, by the same proof of \cite[Theorem~1.1]{Yamamoto1}, it follows. 
\end{rem}
\begin{rem}\label{rem:cl}
By Theorem~\ref{thm:yamamoto} and Girox's work \cite{Giroux1}, we see that a Stallings twist of type $(0,1)$ preserves the supported contact structure. 
However, it changes the knot type of a fibered knot. 
In fact, a single Stallings twist always changes knot types. 
This fact is proved as follows. 
If two fibered knots $K$ and $K'$ with monodromies $f$ and $f'$ are related by a single Stallings twist along $c$, we obtain $t_{c}\circ f=f'$. 
For contradiction, assume that $K=K'$. 
Then, $f$ and $f'$ are conjugate, that is, there is a diffeomorphism $\phi$ on the fiber surface such that $t_{c}\circ f=\phi^{-1} \circ f \circ \phi$. 
This implies $t_{c}$ is a single commutator. 
However, Korkmaz and Ozbagci \cite[Corollary~2.6 and Remarks~(2)]{Korkmaz-Ozbagci} proved that any single Dehn twist cannot be written as a single commutator. 
This is a contradiction. 
This fact is also given by Larson and Meier \cite[Lemma~4.1]{Larson-Meier}. 
\end{rem}
\section{Invariants of plane fields and stabilization heights}\label{sec:Hopf}
Let $(\Sigma, f)$ be an open book decomposition of $\mathbf{S}^{3}$. 
By $\xi_{\Sigma}$, denote the contact structure obtained from $\Sigma$ by Thurston and Winkelnkemper's construction \cite{Thurston-Winkelnkemper} (see also \cite{Giroux1}). 
We call this contact structure the {\it supported contact structure} by $(\Sigma,f)$. 
\par
Fix a parametrization $T\mathbf{S}^3\cong \mathbf{S}^3\times \mathbf{R}^3$. 
For an oriented plane field $\xi$ on $\mathbf{S}^3$, by $v_{\xi}$, we denote a nowhere-zero vector field on $\mathbf{S}^3$ which transversely intersects $\xi$ and is oriented positively with respect to $\xi$. 
Then, $v_{\xi}$ induces a map $\mathbf{S}^3\rightarrow \mathbf{R}^3\setminus \{0\}$. 
Denote the homotopy class of the map by $\mathfrak{a}(\xi)\in \pi_{3}(\mathbf{R}^3\setminus \{0\})\cong \mathbf{Z}$. 
This $\mathfrak{a}$ depends on the parametrization of $T\mathbf{S}^3$. 
It is known that there is a parametrization of $T\mathbf{S}^3$ such that $\mathfrak{a}$ satisfies 
\begin{align*}
&\mathfrak{a}(\xi_{H_{+}})=0, \\
&\mathfrak{a}(\xi_{H_{-}})=-1,
\end{align*}
where $H_{+}$ (resp. $H_{-}$) is the positive (resp. negative) Hopf band. 
In this paper, we call this $\mathfrak{a}$ the {\it Hopf invariant} for plane fields on $\mathbf{S}^3$ and denote it by $H$. 
Note that the classical Hopf invariant is given by $H+1$. 
\par
For an oriented plane field on a $3$-manifold $M$, Gompf \cite{Gompf1} defined a homotopy invariant $d_{3}$. 
In the case $M=\mathbf{S}^{3}$, $d_3$ is essentially equivalent to $H$ 
\footnote{
We comment that Rudolph's invariant $\lambda$ \cite{Rudolph5} of fibered links in $\mathbf{S}^{n}$, called {\it the enhancement to the Milnor number}, is also equivalent to the Hopf invariant in the case $n=3$ (see also \cite{Neumann-Rudolph}). 
In particular, $\lambda(\partial \Sigma)=-H(\xi_{\Sigma})$. 
Hedden \cite[Remark~2.7]{Hedden3} also gave a comment related to this fact. 
}. 
In fact, these invariants satisfy the following properties. 
\begin{align}
&H(\xi_{\Sigma})=-d_{3}(\xi _{\Sigma})-\displaystyle{\frac{1}{2}} \in \mathbf{Z}, \\
&H(\xi_{H_{+}})=0, \label{p1}\\
&H(\xi_{H_{-}})=-1,\\ 
&H(\xi_{\Sigma_{1}\sharp_{p}\Sigma_{2}})=H(\xi_{\Sigma_{1}})+H(\xi_{\Sigma_{2}}), \label{p3} 
\end{align}
where $\Sigma$, $\Sigma_1$ and $\Sigma_2$ are fiber surfaces in $\mathbf{S}^3$, and  $\Sigma_{1}\sharp_{p}\Sigma_{2}$ is a plumbing of $\Sigma_{1}$ and $\Sigma_{2}$ (the same relation holds for Murasugi sum). 
The equations $(\ref{p1})$--$(\ref{p3})$ can be also found in \cite{Hedden2}. 
These invariants are useful in order to estimate stabilization heights. 
\par
Let $\Sigma\subset \mathbf{S}^{3}$ be a fiber surface and $S$ be a common stabilization of $\Sigma$ and the standard disk $D$. 
Define $\alpha_{+}(S)$, $\alpha_{-}(S)$, $\beta_{+}(S)$ and $\beta_{-}(S)$ as follows: 
\begin{itemize}
\item $\alpha_{+}(S)$ (resp. $\alpha_{-}(S)$) is the number of positive (resp. negative) Hopf plumbings needed to obtain $S$ from $\Sigma$, 
\item $\beta_{+}(S)$ (resp. $\beta_{-}(S)$) is the number of positive (resp. negative) Hopf plumbings needed to obtain $S$ from $D$, 
\end{itemize}
Note that $b_{1}(S)-b_{1}(\Sigma)=\alpha_{+}(S)+\alpha_{-}(S)$. 
By computing the Euler numbers, we obtain the following. 
\begin{lem}\label{lem:1}
We have 
\[
\chi (\Sigma)-\alpha_{+}(S)-\alpha_{-}(S)=1-\beta_{+}(S)-\beta_{-}(S)=\chi (S). 
\]
\end{lem}
\begin{lem}\label{lem:2}
We have 
\[
0\leq \beta_{-}(S)=\alpha_{-}(S)-H(\xi_{\Sigma})=-H(\xi_{S}). 
\]
\end{lem}
\begin{proof}
By $(\ref{p1})$--$(\ref{p3})$, we have 
\[
H(\xi_{\Sigma})-\alpha_{-}(S)=H(\xi_{S})=H(\xi_{D})-\beta_{-}(S)=-\beta_{-}(S)\leq 0. 
\]
\end{proof}
By considering the mirror image and Lemmas~\ref{lem:1} and \ref{lem:2}, we obtain Lemma~\ref{lem:3} below (see also \cite[Example~4.4]{Rudolph5} and \cite[Corollary~2.3]{Hedden3}). 
\begin{lem}\label{lem:3}
Let $\overline{\Sigma}$ be the mirror image of $\Sigma$. Then the mirror image $\overline{S}$ of $S$ is a common stabilization of $\overline{\Sigma}$ and $D$. 
Moreover, we obtain 
\begin{align*}
\alpha_{+}(S)&=\alpha_{-}(\overline{S}), \\
\beta_{+}(S)&=\beta_{-}(\overline{S}). 
\end{align*}
In particular, we have 
\[
1-\chi (\Sigma)+H(\xi_{\Sigma})+H(\xi_{\overline{\Sigma}})=0. 
\]
\end{lem}
\begin{proof}
The first and second claims are obvious. 
Let us prove the last claim. 
By Lemma~\ref{lem:2}, we have 
\begin{align*}
H(\xi_{\Sigma})&=\alpha_{-}(S)-\beta_{-}(S),\\
H(\xi_{\overline{\Sigma}})&=\alpha_{-}(\overline{S})-\beta_{-}(\overline{S})=\alpha_{+}(S)-\beta_{+}(S). 
\end{align*}
By Lemma~\ref{lem:1}, we finish the proof. 
\end{proof}
\begin{lem}\label{lem:4}
For any fiber surface $\Sigma\subset \mathbf{S}^{3}$, we obtain 
\[
\max\{H(\xi_{\Sigma}), 0\}+\max\{H(\xi_{\overline{\Sigma}}), 0\}\leq h(\Sigma). 
\]
\end{lem}
\begin{proof}
By Lemma~\ref{lem:2}, we obtain 
\begin{align*}
H(\xi_{\Sigma})&\leq \alpha_{-}(S)\leq \alpha_{-}(S)+\alpha_{+}(S)=b_{1}(S)-b_{1}(\Sigma), \\
H(\xi_{\overline{\Sigma}})&\leq \alpha_{-}(\overline{S})=\alpha_{+}(S)\leq \alpha_{-}(S)+\alpha_{+}(S)=b_{1}(S)-b_{1}(\Sigma), \\
H(\xi_{\Sigma})+H(\xi_{\overline{\Sigma}})&\leq \alpha_{-}(S)+\alpha_{+}(S)=b_{1}(S)-b_{1}(\Sigma). 
\end{align*}
Since the above three inequalities hold for any common stabilization $S$, we have 
\begin{align*}
0&\leq h(\Sigma), \\
H(\xi_{\Sigma})&\leq h(\Sigma), \\
H(\xi_{\overline{\Sigma}})&\leq h(\Sigma), \\
H(\xi_{\Sigma})+H(\xi_{\overline{\Sigma}})&\leq h(\Sigma). 
\end{align*}
These imply 
$
\max\{H(\xi_{\Sigma}), 0\}+\max\{H(\xi_{\overline{\Sigma}}), 0\}\leq h(\Sigma). 
$
\end{proof}
%
\begin{rem}
Typical operations which preserve the fiberness of a surface (or a link) are taking mirror images, Murasugi sums, Stallings twists and cablings. 
The behaviors of the Hopf invariant under taking mirror images and Murasugi sums have been introduced in this section. 
\par 
The behavior of the Hopf invariant under cablings have been described by Hedden \cite[Theorem~2.6]{Hedden3} and Neumann and Rudolph \cite[Corollary~5.3]{Neumann-Rudolph}. 
\par 
On Stallings twists, there are Neumann and Rudolph's work \cite[Theorem~7.1]{Neumann-Rudolph}. 
\end{rem}

%
%
\section{Proof of Theorem~\ref{thm:main}}\label{sec:proof}
In this section, we prove Lemma~\ref{lem:main2}. 
The following lemma has been known (for example, see \cite[p.~84]{Neumann-Rudolph} and \cite[Remark~5.6]{Etnyre-Ozbagci}). 
\begin{lem}\label{lem:main}
Let $\Sigma_{n}\subset \mathbf{S}^3$ be a family of fiber surfaces of the same topological type. 
If $|H(\xi_{\Sigma_n})|\rightarrow \infty$ as $|n|\rightarrow \infty$, then we obtain 
$\displaystyle{\lim_{|n|\rightarrow \infty}}h(\Sigma_n)=\infty$. 
\end{lem}
\begin{proof}
The assumption and Lemma~\ref{lem:3} imply $\max\{H(\xi_{\Sigma_n}), H(\xi_{\overline{\Sigma_n}})\}\rightarrow \infty$ as $|n|\rightarrow \infty$ since $\chi(\Sigma_{n})$ is a constant. 
By Lemma~\ref{lem:4}, we finish the proof. 
\end{proof}
In order to prove Theorem~\ref{thm:main}, we construct a family of fiber surfaces satisfying Lemma~\ref{lem:main2} below. 
\begin{lem}\label{lem:main2}
Let $\Sigma_{n}\subset \mathbf{S}^3$ be a family of fiber surfaces of the same topological type $\Sigma$. 
Assume that the monodromies $\psi_n\colon \Sigma \rightarrow \Sigma$ of $\Sigma_{n}$ satisfy $\psi_n=t^{-n}_{c_{1}}\circ t^{n}_{c_{2}}\circ \psi_{0}$, where $c_1, c_2$ are essential, non-separating simple closed curves on $\Sigma$ and $t_{c_i}$ is the right-handed Dehn twist along $c_i$. 
Then, if $|H(\xi_{\Sigma_n})|\rightarrow \infty$ as $|n|\rightarrow \infty$, we obtain 
\begin{itemize}
\item $\displaystyle{\lim_{|n|\rightarrow \infty}}h(\Sigma_n)=\infty$, 
\item $\operatorname{scl}(\psi_n)\leq\operatorname{cl}(\psi_0)+1$, where we consider the stable commutator length $\operatorname{scl}$ and the commutator length $\operatorname{cl}$ in the mapping class group of a surface containing $\Sigma$ as a subsurface. 
\end{itemize}
\end{lem}
\begin{rem}
It is known that the mapping class group of an orientable surface of genus greater than $2$ is perfect, namely any element in the group is expressed as a product of commutators (see \cite{Harer3, Powell}). 
In Lemma~\ref{lem:main2}, if $g(\Sigma_{n})\leq 2$, and $\psi_n$ and $\psi_0$ cannot be expressed as a product of commutators, we define the commutator lengths by $\infty$. 
\end{rem}
\begin{proof}[Proof of Lemma~\ref{lem:main2}]
The first claim follows from Lemma~\ref{lem:main}. 
Let us prove the second claim. 
Let $c_{1}$ and $c_{2}$ be the curves given in the assumption. 
Since these curves are non-separating, there is a diffeomorphism $g\colon \Sigma \rightarrow \Sigma$ fixing the boundary such that $g(c_{1})=c_{2}$ (for example, see \cite[Proposition~12.7]{knot-gtm}). 
Hence, 
\begin{align*}
\operatorname{scl}(\psi_n)=\operatorname{scl}(t^{-n}_{c_{1}}\circ t^{n}_{c_{2}}\circ \psi_{0}) &\leq \operatorname{cl}(t^{-n}_{c_{1}}\circ t^{n}_{g(c_{1})}\circ \psi_{0})\\
&=\operatorname{cl}(t^{-n}_{c_{1}}\circ g\circ t^{n}_{c_{1}}\circ g^{-1}\circ \psi_{0})\\
&=\operatorname{cl}([t^{-n}_{c_{1}}, g]\circ \psi_{0})\\
&\leq\operatorname{cl}(\psi_{0})+1, 
\end{align*}
where $[t^{-n}_{c_{1}}, g]$ is the commutator of $t^{-n}_{c_{1}}$ and $g$. 
\end{proof}
%

\begin{proof}[Proof of Theorem~\ref{thm:main}]
Let $F_{0}\subset \mathbf{S}^3$ be the surface depicted in Figure~\ref{fig:monodromy2}. 
The surface $F_{0}$ is a fiber surface since it is a minimal Seifert surface of the fibered knot $\partial F_{0}=6_3$. 
More directly, we can check that $F_{0}$ is obtained from a disk by plumbing $2$ positive Hopf bands and $2$ negative Hopf bands (see the right in Figure~\ref{fig:monodromy2}). 
The monodromy $f_{0}$ of $F_{0}$ is given by 
\[
f_{0}=t^{-1}_d \circ t_b \circ t^{-1}_c\circ t_a, 
\]
where $a,b,c$ and $d$ are the curves depicted in the left pictures in Figure~\ref{fig:monodromy}, and $t_a$, $t_b$, $t_c$ and $t_d$ are the right-handed Dehn twists along the corresponding curves. 
\par 
Let $c'_{1}$ and $c'_{2}$ be the disjoint non-separating simple closed curves on $F_{0}$ depicted in Figure~\ref{fig:monodromy}. 
By twisting along $c'_1$ of order $n$ and along $c'_2$ of order $-n$, we obtain a new fiber surface in $M_{c'_1\cup c'_2}(1/n, -1/n)\cong \mathbf{S}^3$, where $M_{c'_1\cup c'_2}(1/n, -1/n)$ is the $3$-manifold obtained by $(1/n, -1/n)$-surgery along $c'_1\cup c'_2$. 
Denote the fiber surface by $F_{n}$. 
Then the monodromy $f_{n}$ of $F_{n}$ is given by 
\[
f_{n}=t^{-n}_{c'_{1}}\circ t^{n}_{c'_{2}}\circ f_0. 
\]
In \cite[Section~6]{Abe-Tagami3}, we compute that $d_{3}(\xi_{F_{n}})=-n^{2}-n+3/2$. 
Hence, $H(\xi_{F_{n}})=n^{2}+n-2$. 
This implies that $H(\xi_{F_{n}})\rightarrow \infty$ as $|n|\rightarrow \infty$. 
\par 
Let $F$ be a fixed fiber surface of genus greater than $0$. Denote the monodromy of $F$ by $f$. 
Let $\Sigma_{n}$ be a boundary sum of $F_n$ and $F$. 
Note that $\Sigma_n$ is a fiber surface of genus greater than $2$ and has the topological type $\Sigma_{0}$. 
Denote the corresponding monodromy of $\Sigma_n$ by $\psi_n=t^{-n}_{c'_{1}}\circ t^{n}_{c'_{2}}\circ f_{0}\circ f$. 
By the above calculation, we have $H(\xi_{\Sigma_{n}})=H(\xi_{F_{n}})+H(\xi_{F})\rightarrow \infty$ as $|n|\rightarrow \infty$. 
Hence, the family $\{\Sigma_{n}\}_{n\in \mathbf{Z}}$ satisfies the condition of Lemma~\ref{lem:main2}. 
Then, we obtain 
\[
\operatorname{scl}(\psi_{n})\leq \operatorname{cl}(\psi_0)+1, 
\]
where we consider the commutator lengths in the mapping class group of a surface $S$ containing $\Sigma_0$ as a subsurface. 
We see that $\operatorname{cl}(\psi_0)$ is maximal if $S=\Sigma_0$, and we have the second condition of Theorem~\ref{thm:main}. 
\begin{figure}[h]
\centering
\includegraphics[scale=0.8]{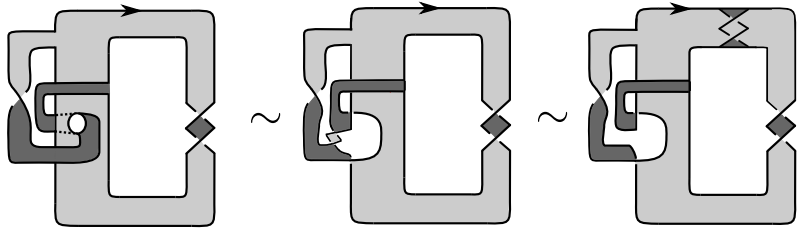}
\caption{The fiber surface $F_{0}$ of $6_3$}
\label{fig:monodromy2}
\end{figure}
\begin{figure}[h]
\centering
\includegraphics[scale=0.9]{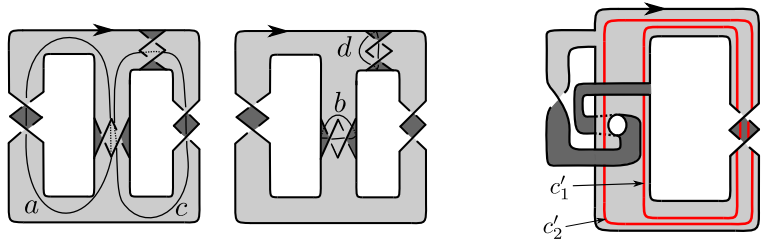}
\caption{(color online) The curves $a,b,c,d$, and $c'_1, c'_2$ on $F_{0}$}
\label{fig:monodromy}
\end{figure}
\end{proof}
The fibered knots $\partial F_{n}$ in Proof of Theorem~\ref{thm:main} are constructed from a ``compatible annulus presentation" of $6_3$ by annulus twists (for definition, see \cite[Section~5]{Abe-Tagami2}). 
By the same way, it seems that we can construct many families of fiber surfaces which satisfy the assumption of Lemma~\ref{lem:main2}. 
%
%
\section{Discussion}\label{sec:discussion}
In Baader and Misev's proof of Theorem~\ref{thm:baader-misev}, they used the unboundedness of stable commutator lengths of monodromies. 
Because $\operatorname{scl}(\psi_{n})$ of Lemma~\ref{lem:main2} is bounded, Baader and Misev's technique 
cannot be used to prove Theorem~\ref{thm:main}. 
\par
On the other hand, in Proof of Theorem~\ref{thm:main} (and Lemma~\ref{lem:main2}), we use the unboundedness of the Hopf invariants. 
In general, the Hopf invariants are not preserved under Stallings twists. 
However, Stallings twists of type $(0,1)$ preserve the Hopf invariant (see Section~\ref{sec:stallings} and Remark~\ref{rem:cl}). 
In particular, we cannot apply Lemma~\ref{lem:main2} to Baader and Misev's example \cite[Section~2.2]{Baader-Misev}, which is given by a Stallings twist of type $(0,1)$. 
%
%
\par
\ 
\par
\noindent{\bf Acknowledgements: }
We would like to thank Susumu Hirose and Naoyuki Monden for helpful comments on Remark~\ref{rem:cl}. 
The author was supported by JSPS KAKENHI Grant numbers JP16H07230 and JP18K13416. 
\bibliographystyle{amsplain}
\bibliography{mrabbrev,tagami}

\providecommand{\bysame}{\leavevmode\hbox to3em{\hrulefill}\thinspace}
\providecommand{\MR}{\relax\ifhmode\unskip\space\fi MR }
\providecommand{\MRhref}[2]{%
  \href{http://www.ams.org/mathscinet-getitem?mr=#1}{#2}
}
\providecommand{\href}[2]{#2}
\begin{thebibliography}{10}

\bibitem{Abe-Tagami3}
T.~Abe and K.~Tagami, \emph{Addendum to ``{F}ibered knots with the same
  0-surgery and the slice-ribbon conjecture"}, RIMS K{\^o}ky{\^u}roku
  \textbf{1960} (2016), 18--36.

\bibitem{Abe-Tagami2}
\bysame, \emph{Fibered knots with the same 0-surgery and the slice-ribbon
  conjecture}, Math. Res. Lett. \textbf{23} (2016), no.~2, 303--323.
  \MR{3512887}

\bibitem{Baader-Misev}
S.~Baader and F.~Misev, \emph{On the stabilization height of fiber surfaces in
  {$S^3$}}, J. Knot Theory Ramifications \textbf{27} (2018), no.~3, 1840001, 8.
  \MR{3786775}

\bibitem{Etnyre-Ozbagci}
J.~B. Etnyre and B.~Ozbagci, \emph{Invariants of contact structures from open
  books}, Trans. Amer. Math. Soc. \textbf{360} (2008), no.~6, 3133--3151.
  \MR{2379791}

\bibitem{gabai2}
D.~Gabai, \emph{The {M}urasugi sum is a natural geometric operation},
  Low-dimensional topology ({S}an {F}rancisco, {C}alif., 1981), Contemp. Math.,
  vol.~20, Amer. Math. Soc., Providence, RI, 1983, pp.~131--143. \MR{718138
  (85d:57003)}

\bibitem{gabai1}
\bysame, \emph{The {M}urasugi sum is a natural geometric operation. {II}},
  Combinatorial methods in topology and algebraic geometry ({R}ochester,
  {N}.{Y}., 1982), Contemp. Math., vol.~44, Amer. Math. Soc., Providence, RI,
  1985, pp.~93--100. \MR{813105 (87f:57003)}

\bibitem{Giroux1}
E.~Giroux, \emph{G\'eom\'etrie de contact: de la dimension trois vers les
  dimensions sup\'erieures}, Proceedings of the {I}nternational {C}ongress of
  {M}athematicians, {V}ol. {II} ({B}eijing, 2002), Higher Ed. Press, Beijing,
  2002, pp.~405--414. \MR{1957051}

\bibitem{Giroux-Goodman}
E.~Giroux and N.~Goodman, \emph{On the stable equivalence of open books in
  three-manifolds}, Geom. Topol. \textbf{10} (2006), 97--114 (electronic).
  \MR{2207791}

\bibitem{Gompf1}
R.~Gompf, \emph{Handlebody construction of {S}tein surfaces}, Ann. of Math. (2)
  \textbf{148} (1998), no.~2, 619--693. \MR{1668563}

\bibitem{Goodman1}
N.~Goodman, \emph{Contact structures and open books}, ProQuest LLC, Ann Arbor,
  MI, 2003, Thesis (Ph.D.)--The University of Texas at Austin. \MR{2705496}

\bibitem{harer2}
J.~Harer, \emph{P{ENCILS} {OF} {CURVES} {ON} 4-{MANIFOLDS}}, ProQuest LLC, Ann
  Arbor, MI, 1979, Thesis (Ph.D.)--University of California, Berkeley.
  \MR{2628695}

\bibitem{Harer1}
\bysame, \emph{How to construct all fibered knots and links}, Topology
  \textbf{21} (1982), no.~3, 263--280. \MR{649758}

\bibitem{Harer3}
\bysame, \emph{The second homology group of the mapping class group of an
  orientable surface}, Invent. Math. \textbf{72} (1983), no.~2, 221--239.
  \MR{700769}

\bibitem{Hedden3}
M.~Hedden, \emph{Some remarks on cabling, contact structures, and complex
  curves}, Proceedings of {G}\"okova {G}eometry-{T}opology {C}onference 2007,
  G\"okova Geometry/Topology Conference (GGT), G\"okova, 2008, pp.~49--59.
  \MR{2509749}

\bibitem{Hedden2}
\bysame, \emph{Notions of positivity and the {O}zsv\'ath-{S}zab\'o concordance
  invariant}, J. Knot Theory Ramifications \textbf{19} (2010), no.~5, 617--629.
  \MR{2646650 (2011j:57020)}

\bibitem{Korkmaz-Ozbagci}
M.~Korkmaz and B.~Ozbagci, \emph{Minimal number of singular fibers in a
  {L}efschetz fibration}, Proc. Amer. Math. Soc. \textbf{129} (2001), no.~5,
  1545--1549. \MR{1713513}

\bibitem{Larson-Meier}
K.~Larson and J.~Meier, \emph{Fibered ribbon disks}, J. Knot Theory
  Ramifications \textbf{24} (2015), no.~14, 1550066, 22. \MR{3434552}

\bibitem{knot-gtm}
W.~B.~R. Lickorish, \emph{An introduction to knot theory}, Graduate Texts in
  Mathematics, vol. 175, Springer-Verlag, New York, 1997. \MR{1472978
  (98f:57015)}

\bibitem{Melvin-Morton}
P.~M. Melvin and H.~R. Morton, \emph{Fibred knots of genus {$2$} formed by
  plumbing {H}opf bands}, J. London Math. Soc. (2) \textbf{34} (1986), no.~1,
  159--168. \MR{859157}

\bibitem{Neumann-Rudolph}
W.~D. Neumann and L.~Rudolph, \emph{Difference index of vectorfields and the
  enhanced {M}ilnor number}, Topology \textbf{29} (1990), no.~1, 83--100.
  \MR{1046626}

\bibitem{Powell}
J.~Powell, \emph{Two theorems on the mapping class group of a surface}, Proc.
  Amer. Math. Soc. \textbf{68} (1978), no.~3, 347--350. \MR{494115}

\bibitem{Rudolph5}
L.~Rudolph, \emph{Isolated critical points of mappings from {${\bf R}\sp 4$} to
  {${\bf R}\sp 2$} and a natural splitting of the {M}ilnor number of a
  classical fibered link. {I}. {B}asic theory; examples}, Comment. Math. Helv.
  \textbf{62} (1987), no.~4, 630--645. \MR{920062}

\bibitem{Stallings}
J.~Stallings, \emph{Constructions of fibred knots and links}, Algebraic and
  geometric topology ({P}roc. {S}ympos. {P}ure {M}ath., {S}tanford {U}niv.,
  {S}tanford, {C}alif., 1976), {P}art 2, Proc. Sympos. Pure Math., XXXII, Amer.
  Math. Soc., Providence, R.I., 1978, pp.~55--60. \MR{520522}

\bibitem{Thurston-Winkelnkemper}
W.~Thurston and H.~Winkelnkemper, \emph{On the existence of contact forms},
  Proc. Amer. Math. Soc. \textbf{52} (1975), 345--347. \MR{0375366}

\bibitem{Yamamoto1}
R.~Yamamoto, \emph{Stallings twists which can be realized by plumbing and
  deplumbing {H}opf bands}, J. Knot Theory Ramifications \textbf{12} (2003),
  no.~6, 867--876. \MR{2008884}

\bibitem{Yamamoto2}
\bysame, \emph{Open books supporting overtwisted contact structures and the
  {S}tallings twist}, J. Math. Soc. Japan \textbf{59} (2007), no.~3, 751--761.
  \MR{2344826}

\end{thebibliography}
\end{document}